\documentclass[a4paper,11pt]{amsart}
\usepackage{amssymb,amsxtra}

\theoremstyle{plain}
\newtheorem{theorem}{Theorem}[section]
\newtheorem{lemma}[theorem]{Lemma}
\newtheorem{corollary}[theorem]{Corollary}
\theoremstyle{remark}
\newtheorem{example}[theorem]{Example}
\newtheorem{remark}[theorem]{Remark}

\newcommand{\R}{\mathbf{R}}
\newcommand{\C}{\mathbf{C}}

\newcommand{\Q}{\mathbf{Q}}
\renewcommand{\P}{\mathbf{P}}

\DeclareMathOperator{\re}{Re}
\DeclareMathOperator{\im}{Im}

\begin{document}
\title[Elimination of parasitic solutions]{Elimination of parasitic solutions in theory of flexible polyhedra}
\author{I.~Kh. Sabitov}
\address{Faculty of Mechanics and Mathematics \\
         Lomonosov Moscow State University \\
         Leninskie Gory, Moscow GSP-1,119234, Russia}
\email{isabitov@mail.ru}
\author{D.~A. Stepanov}
\address{The Department of Mathematical Modelling \\
         Bauman Moscow State Technical University \\
         2-ya Baumanskaya ul. 5, Moscow 105005, Russia}
\email{dstepanov@bmstu.ru}
\date{}

\begin{abstract}
The action of the rotation group $SO(3)$ on systems of $n$ points in the $3$-dimensional Euclidean
space $\mathbf{R}^3$ induces naturally an action of $SO(3)$ on $\mathbf{R}^{3n}$. In the present paper
we consider the following question: do there exist $3$ polynomial functions $f_1$, $f_2$, $f_3$ on
$\mathbf{R}^{3n}$ such that the intersection of the set of common zeros of $f_1$, $f_2$, and
$f_3$ with \emph{each} orbit of $SO(3)$ in $R^{3n}$ is nonempty and finite? Questions of this kind arise 
when one is interested in \emph{relative} motions of a given set of $n$ points, i.~e., when one wants to
exclude the local motions of the system of points as a rigid body. An example is the problem of 
deciding whether a given polyhedron is non-trivially flexible. We prove that such functions do exist. 
To get a necessary system of equations $f_1=0$, $f_2=0$, $f_3=0$, we show how starting by choice of 
a hypersurface in $\mathbf{CP}^{n-1}$ containing no conics, no lines, and no real points one can find such 
a system.
\end{abstract}

\maketitle

\section{Introduction}

For many problems in mathematics and physics one needs to study some properties of a system of
points which depend only on the distances between some or all the pairs of points. In such a case
the corresponding properties should be invariant under the motion of the points as particles of a rigid body
or under an orthogonal change of coordinates. At the same time such a motion leads to 
different positions of the system in space, thus posing a question on geometric or physical identity
of the new system with the initial one. For example, if we look for a simplicial polyhedron with
given combinatorial structure and known edge lengths, then the coordinates of its vertices can be 
obtained as solutions of the system of equations
\begin{equation}\label{E:lengths}
(x_i-x_j)^2+(y_i-y_j)^2+(z_i-z_j)^2=l_{ij}^2,
\end{equation}
where the pair $(i,j)$ varies over all edges of the polyhedron with known length $l_{ij}$. However,
having found two different solutions of these equations, we do not know a priori whether they give
polyhedra that differ only by a continuous motion as rigid bodies or these polyhedra are isometric and
noncongruent ones. Note that we call two polyhedra \emph{isometric} if they have the same combinatorial
structure and the corresponding edges are equal, but the polyhedra are not necessarily related by
an isometry of the ambient Euclidean space. This difficulty is usually formally avoided by saying that 
``solutions are considered up to a motion as a rigid body'' or using the formal operation of factorization of 
the set of polyhedra by the isometry group of the affine Euclidean space $\R^3$. But this trick does not give 
a direct answer to the question whether two close solutions are related by an isometry of $\R^3$ or by a 
non-trivial \emph{flexion} with rigid faces (that is by a continuous family of solutions of \eqref{E:lengths};
one uses the term \emph{bending} too) and in practice one has to somehow check this. In his papers
(see, e.g., \cite{Sabitov98a}, \cite{Sabitov98b}, \cite{Sabitov11}) the first author suggested to replace 
the factorization operation by adding to system \eqref{E:lengths} several new equations which should automatically exclude the ``parasitic'' solutions. For instance, it is always possible to suppose that the 
origin of coordinates is placed in the center of mass of the given system of material points and then to 
adjoin to \eqref{E:lengths} three new equations
\begin{equation}\label{E:cofm}
\sum_j x_j=0, \quad \sum_j y_j=0, \quad \sum_j z_j=0
\end{equation}
(all the masses are supposed to be equal), thus eliminating the parasitic solutions related
to a parallel translation. Further, to exclude solutions that are obtained from a given one by
a rotation around the origin (that is, by a transformation from the group $SO(3)$), we could
impose an additional condition that the system of coordinates is chosen in accordance with the principal
axes of inertia of given points. This choice of coordinates is equivalent to adjoining three more equations
$$\sum_j x_jy_j=0, \quad \sum_j y_jz_j=0, \quad \sum_j x_jz_j=0$$
to system \eqref{E:lengths}. But these new equations work well only under the assumption that
$$\sum_j x_{j}^{2}\ne \sum_j y_{j}^{2}, \quad \sum_j y_{j}^{2}\ne \sum_j z_{j}^{2}, \quad
\sum_j x_{j}^{2}\ne \sum_j z_{j}^{2},$$
otherwise the system of points admits rotations around one of the coordinate axes. At the same
time we look for equations that would exclude rotations for \emph{any} initial position of
the points. Another variant of screening for "true" flexion consists in fixation of
one of the non-degenerate triangular faces. In turn, this method is not universal also since for some 
collections of points the given face can be degenerated so we should know in advance the existence
of a non-degenerated face.

One more example where one needs to somehow fix a system of points is the problem of recovering
of positions of points from the given distances between each pair of them, as in the well known
Lennard-Jones Problem from physics. In Lennard-Jones Problem one seeks a minimum of
a functional which is a function of the distances between points. If, say, the necessary distances
have been determined, then to recover the positions of points of the initial system one has
to choose a unique configuration which suits in the best way some additional requirements. Thus again
one has to look for solutions of a system analogous to \eqref{E:lengths} and eliminate the
repetition of copies. Another important example is the classical $n$-body problem, where
the quotient $\R^{3n}/SO(3)$ (referred to as the \emph{shape space}) gives some interesting
insights, see \cite{Montgomery}.

Now let us formulate our problem and results more precisely. Let $P_1,\dots,P_n$ be a system of 
points in the affine space $\R^m$ with the standard Euclidean structure. Some or even all of the points of
the system can coincide, that is, possess equal coordinates. Let $G$ be a group of isometries of $\R^m$.
If $\rho$ is an element of $G$, then by $\rho P_j$ we denote the image of the point $P_j$ under the action 
of $\rho$. Note that the system $\{P_j\}_{j=1}^{n}$ can be represented by a single point $M_0$ in the
Euclidean space $\R^{mn}=\R^m\times\cdots\times\R^m$, thus we have a natural diagonal
action of the group $G$ on $\R^{mn}$. The orbit of the point $M_0$ in $\R^{mn}$ under this action
will sometimes be referred to also as the orbit of the system $\{P_j\}_{j=1}^{n}$.

Let $f=(f_1,\ldots,f_k)$ be a collection of real valued functions on $\R^{mn}$. Keeping in mind the 
splitting of $\R^{mn}$ as a product of $n$ copies of $\R^m$, we shall consider each of the functions 
$f_1,\ldots,f_k$ also as a function of $n$ points in $\R^m$. We shall say that the system of points
$\{P_j\}_{j=1}^{n}\subset \R^m$ \emph{admits a reduction} by $G$ with respect to $f$, if there exists 
an element $\rho\in G$ such that $f_l(\rho P_1,\dots,\rho P_n)=0$ for each $l=1,\ldots,k$. In other words, 
the variety
$$X_f=\{M\in\R^{mn}\,|\,f_1(M)=\dots=f_k(M)=0\}$$
has a nonempty intersection with the orbit of the system $\{P_j\}_{j=1}^{n}$. We shall say that the
system $\{P_j\}_{j=1}^{n}$ \emph{is fixed} by $f$ with respect to $G$, if the orbit of 
$\{P_j\}_{j=1}^{n}$ has no more than finite number of intersection points with the variety $X_f$.  
For example, for $m=3$ and $G$ the group of parallel translations of $\R^3$, \emph{each} system 
of points $\{P_j\}_{j=1}^{n}\subset \R^3$ admits a reduction with respect to and is fixed by the
functions on the left hand side of \eqref{E:cofm}. Instead of the \emph{collection of functions} $f$, 
we sometimes say that a system of points admits a reduction with respect to (or is fixed by) 
the \emph{system of equations} $f_1=\cdots=f_k=0$.

In Section~\ref{S:dim2} we study the case $m=2$, that is, the case of the Euclidean plane $\R^2$,
and $G=SO(2)$ the group of plane rotations around the origin. We consider the canonical representation
of $SO(2)$ as the group of orthogonal $2$ by $2$ matrices with determinant one which act on
$\R^2$ by left multiplication. We show that \emph{each} system of points 
$\{P_j\}_{j=1}^{n}=\{(x_{j}^{\circ},y_{j}^{\circ})\}_{j=1}^{n}$
on the plane admits a reduction with respect to and in the same time is fixed by the function
$$A(x_1,y_1,\dots,x_n,y_n)=(x_n y_n)^{2n-1}+\cdots+(x_2 y_2)^3+x_1 y_1$$
proposed by our untimely deceased colleague A.\,V. Astrelin. Note that if at least one of the points 
$P_j$ of a system $\{P_j\}_{j=1}^{n}$ is different from the origin, then the orbit of such a system in 
$\R^{2n}$ is homeomorphic to the circle $S^1\simeq SO(2)$.

Now let 
$\{P_j\}_{j=1}^{n}=\{(x_{j}^{\circ},y_{j}^{\circ},z_{j}^{\circ})\}_{j=1}^{n}$
be a system of points in the $3$-dimensional Euclidean space $\R^3$, and $G=SO(3)$ be the group 
of space rotations around the origin in its canonical representation. In Section~\ref{S:dim3}, we prove 
our main result.

\begin{theorem}\label{T:main}
Let $F(w_1,\ldots,w_n)$ be a homogeneous polynomial of some degree $2d$ such that the hypersurface
$X_F$ defined by $F=0$ in the \emph{complex} projective space $\C\P^{n-1}$ contains no conics, no lines,
and no real points. Let
$$H(w_1,\dots,w_n)=w_{1}^{p_1}+w_{2}^{p_2}+\dots+w_{n}^{p_n}, \quad p_1>p_2>\dots>p_n,$$
where $w_j=x_j+iy_j$ are complex variables. Then \emph{each} system of $n$ points in $\R^3$ admits a
reduction by the group $SO(3)$ with respect to the $3$ functions
$$f_1=\re F, \: f_2=\im F, \: f_3=\im H,$$
\emph{considered as functions on} $\R^{3n}$, and at the same time is fixed by $f_1,f_2,f_3$
with respect to $SO(3)$.
\end{theorem}

For the diagonal action of $SO(3)$ on $\R^{3n}$, the orbit of the system $\{P_j\}_{j=1}^{n}$ 
(that is, the orbit of $M_0$ in $\R^{3n}$) is homeomorphic to the real $3$-dimensional projective space
$\R\P^3\simeq SO(3)$ whenever the points $O$, $P_1$, $\dots$, $P_n$ are not collinear, where $O$ is 
the origin (see Lemma~\ref{L:nondegorbit}). If the points $O,P_1,\dots,P_n$ are collinear but at least one 
of the points $P_1,\dots,P_n$ is different from the origin, then the orbit is homeomorphic to the 
$2$-dimensional sphere $S^2$, finally, the orbit is reduced to a single point if all the points coincide with 
the origin.

The full group of affine isometries of $\R^3$ is generated by the subgroup of parallel translations and
by the full orthogonal group $O(3)$. But the orbit of a system of points under the action of $O(3)$
is a union of no more than $2$ orbits of $SO(3)$, thus we get the following corollary.

\begin{corollary}\label{C:fullfix}
Let $G$ be the full group of affine isometries of $\R^3$. Then each system of points in $\R^3$ admits 
a reduction by $G$ with respect to $6$ functions $3$ of which are defined in \eqref{E:cofm} and the
other $3$ in Theorem~\ref{T:main}, and at the same time is fixed by these $6$ functions.
\end{corollary}

Corollary~\ref{C:fullfix} can be applied in the theory of flexible polyhedra. Assume now that
$\{P_j\}_{j=1}^{n}$ are the vertices of a polyhedron $P$ in $\R^3$ and we are interested whether
polyhedron $P$ is flexible. Consider the system of algebraic equations \eqref{E:lengths} fixing the 
lengths of edges of $P$ and adjoin to this system $6$ more equations, namely, $3$ equations
\eqref{E:cofm} and $3$ more equations $f_1=f_2=f_3=0$ for $f_1,f_2,f_3$ from Theorem~\ref{T:main}. 
Let us call so obtained system the \emph{extended system} of the polyhedron $P$.

\begin{corollary}\label{C:rigpoly}
(a) A polyhedron $P$ with coordinates of vertices represented by a point 
$M_0=(x_{1}^{\circ},y_{1}^{\circ},z_{1}^{\circ},\dots,x_{n}^{\circ},y_{n}^{\circ},z_{n}^{\circ})$
in $\R^{3n}$ satisfying the extended system is not flexible if and only if there exists a neighborhood $U$ 
of $M_0$ in $\R^{3n}$ such that $M_0$ is the only solution to the extended system of equations in $U$. 
(b) All polyhedra isometric to the given polyhedron $P$ are not flexible if and only if the extended system
of equations of $P$ has only finite number of solutions.
\end{corollary} 

To facilitate the correct understanding of Corollary~\ref{C:rigpoly} recall that, for example,
together with the flexible Bricard octahedron of the first type there exists an isometric to it
continuously rigid octahedron. Thus, in the neighborhoods of the corresponding points in $\R^{18}$
the solution set to the extended system has different structure.

\begin{remark}
As the reader has perhaps noticed, the functions $f_1$, $f_2$, $f_3$ in Theorem~\ref{T:main}
depend only on the \emph{projections} of the points $P_j$ to the coordinate plane $Oxy$. Evidently 
an analogue of Theorem~\ref{T:main} is valid for the cases when the planes $Oxz$ and $Oyz$ 
are selected.
\end{remark}

\begin{remark}
Existence of hypersurfaces $X_F$ with the properties stated in Theorem~\ref{T:main} follows from
general results of algebraic geometry. For example, we can quote the main result of
\cite{Clemens} which says that a \emph{generic} hypersurface of degree $r$ in $\C\P^n$ has
no \emph{rational curves} as soon as $r\geq 2n-1$. This gives the estimate $2d\geq 2n-3$ in
our case. On the other hand, since we are interested only in lines and conics, Clemens' estimate
can be improved to $r>\frac{1}{2}(3n+2)$ (\cite[Theorem~1.1]{Furukawa}). We insist on even
degree $2d$ to be able to get a hypersurface with no real points. Indeed, for this we can take a
generic hypersurface in a small neighborhood (in the space parameterizing degree $2d$ hypersurfaces 
in $\C\P^n$) of the Fermat hypersurface
$$w_{1}^{2d}+w_{2}^{2d}+\dots+w_{n}^{2d}=0.$$
From computational perspective, it is desirable to choose polynomial $F$ with, let us say,
``simple'' coefficients, i.e., the hypersurface $X_F$ to be defined over $\Q$ or some of its finite
algebraic extensions. Definitely, having chosen a particular hypersurface that seems likely to fulfill
our conditions, it should be possible to check by a computer whether it really does. But we are not 
aware of any general method that would give such a hypersurface for all $n$.
\end{remark}

We know two published works where a problem similar to ours was considered. The first is paper
\cite{Connelly} of R. Connelly, where at the end the author presents an equation restricting
a motion of a system of points in $\R^3$ in such a way that if their motion is a part of a rigid motion
of all of $\R^3$, then the system in reality is fixed. This result is, however, very different from ours because 
Connelly examines not the question of flexibility of a given polyhedron, but only the question
of triviality of a \emph{given flexion} of a polyhedron (or a system of points).
In \cite[Section~2.3.4, Exercise (d')]{Gromov}, M. Gromov proposes to the reader to prove that if
an algebraic foliation of $\R^n$ into codimension $q$ leaves is given, then there exists a
$q$-dimensional algebraic subset of $\R^n$ intersecting all the leaves of the foliation. The main idea,
as can be guessed from the preceding discussion in \cite{Gromov}, is to choose a sufficiently generic
polynomial $f$ on $\R^n$ and to consider the locus of critical points of all the restrictions of $f$ to each
leaf of the foliation. Our results do not, however, follow from this exercise since, in Gromov's approach, 
it is not clear why that $q$-dimensional algebraic subset must be given exactly by $q$ equations 
(i.e., is a complete intersection in the language of algebraic geometry); also, the question whether
the number of intersection points of the algebraic subset with each leaf is finite (the question
of fixation in our terminology) is not discussed in \cite{Gromov}.

\section{Reduction and fixation of a system of points in the plane}\label{S:dim2}

A function $A\colon\R^{2n}\to \R$, where
$$A(x_1,y_1,\dots,x_n,y_n)=(x_n y_n)^{2n-1}+\cdots+(x_2 y_2)^3+x_1 y_1,$$
will be called the \emph{plane Astrelin function}.

\begin{theorem}
Each system $\{P_j\}_{j=1}^{n}$ of points in the plane admits a reduction by the group $SO(2)$ with 
respect to the plane Astrelin function $A$ and at the same time is fixed by $A$ with respect to $SO(2)$.
\end{theorem}
\begin{proof}
If all the points of the given system are concentrated in the origin, then the orbit of the system under
the action of $SO(2)$ is reduced to a single point and the theorem is obviously true in this case.
Thus in the rest of the proof we assume that at least one of the points $P_j$, $j=1,\dots,n$, is different
from the origin.

The reduction with respect to the Astrelin function is easy to prove. If $\sigma$ is a rotation by
$90^\circ$ and $P_j=(x_{j}^{\circ},y_{j}^{\circ})$, then $\sigma P_j=(-y_{j}^{\circ},x_{j}^{\circ})$
and
$$A(\sigma P_1,\dots,\sigma P_n)=-A(P_1,\dots,P_n).$$
Since the group $SO(2)$ is pathwise connected, there exists an angle $\varphi$,
$0\leq \varphi\leq 90^\circ$, such that the rotation $\rho_\varphi$ by this angle reduces the system:
$$A(\rho_\varphi P_1,\dots,\rho_\varphi P_n)=0.$$

The proof of fixation starts with the following parametrization of $SO(2)$:
$$SO(2)=\left\{
\begin{pmatrix}
\cos\theta & -\sin\theta \\
\sin\theta & \cos\theta
\end{pmatrix}\, \Big|\,
\theta\in\R \right\}.
$$
The restriction of the plane Astrelin function to an orbit is given by the function
\begin{align*}
A(\theta) &= [\frac{1}{2}((x_{n}^{\circ})^{2}-(y_{n}^{\circ})^{2})\sin 2\theta
      +x_{n}^{\circ} y_{n}^{\circ}\cos 2\theta]^{2n-1}+ \cdots+ \\
&+ [\frac{1}{2}((x_{2}^{\circ})^{2}-(y_{2}^{\circ})^{2})\sin 2\theta
      +x_{2}^{\circ} y_{2}^{\circ}\cos 2\theta]^3+ \\
&+ [\frac{1}{2}((x_{1}^{\circ})^{2}-(y_{1}^{\circ})^{2})\sin 2\theta
                    +x_{1}^{\circ} y_{1}^{\circ}\cos 2\theta],
\end{align*}
which is analytic as a function of $\theta$. If the equation $A(\theta)=0$ had infinite set of solutions
on the interval $[0,2\pi]$, then the function $A(\theta)$ would vanish identically, and thus all of its
Fourier coefficients would vanish too. The Fourier coefficient with the biggest number can be found
from the expansion into Fourier series of the first summand
$$A_{2n-1}(\theta)=[\frac{1}{2}((x_{n}^{\circ})^{2}-(y_{n}^{\circ})^{2})\sin 2\theta
      +x_{n}^{\circ} y_{n}^{\circ}\cos 2\theta]^{2n-1}.$$
We can assume that $P_n\ne O$, thus this first summand is nonzero.
Consider a representation
$$A_{2n-1}(\theta)=\Big(\alpha_n \frac{e^{2i\theta}+e^{-2i\theta}}{2}+
\beta_n\frac{e^{2i\theta}-e^{-2i\theta}}{2i}\Big)^{2n-1}=
(\gamma_n e^{2i\theta}+\bar{\gamma}_n e^{-2i\theta})^{2n-1},$$
where
$$\alpha_n= x_{n}^{\circ} y_{n}^{\circ},\: \beta_n= \frac{1}{2}((x_{n}^{\circ})^2-(y_{n}^{\circ})^2),\:
\gamma_n=\frac{\alpha_n-i\beta_n}{2}.$$
Note that $\gamma_n\ne 0$. Further, by Newton binomial formula
$$(\gamma_n e^{2i\theta}+\bar{\gamma}_n e^{-2i\theta})^{2n-1}=
\sum_{k=0}^{2n-1} C_{2n-1}^{k} \gamma_{n}^{k}\bar{\gamma}_{n}^{2n-1-k} e^{i(4k-4n+2)\theta}.$$
The leading Fourier coefficient $a_{4n-2}+ib_{4n-2}$ is calculated by the formula
$$a_{4n-2}+ib_{4n-2}=\frac{1}{\pi}\int_{0}^{2\pi}
\sum_{k=0}^{2n-1} C_{2n-1}^{k} \gamma_{n}^{k}\bar{\gamma}_{n}^{2n-1-k} e^{i(4k)\theta}\,d\theta,$$
and from this entire sum only the summand with $k=0$ survives. Therefore,
$$a_{4n-2}+ib_{4n-2}=2\bar{\gamma}_{n}^{2n-1}\ne 0,$$
that is $A(\theta)\not\equiv 0$. We get a contradiction that proves the theorem.
\end{proof}

\begin{example}
Consider the square with initial position of vertices $P_1=(1,0)$, $P_2=(0,1)$, $P_3=(-1,0)$, $P_4=(0,-1)$.
Introducing polar coordinates by the formulae
$$x_j=\cos\Big(\varphi+(j-1)\frac{\pi}{2}\Big), \quad y_j=\sin\Big(\varphi+(j-1)\frac{\pi}{2}\Big), \quad
j=1,2,3,4,$$
we reduce Astrelin's equation 
$$x_1y_1+(x_2y_2)^3+(x_3y_3)^5+(x_4y_4)^7=0$$
to the equation
$$\frac{1}{2}\sin 2\varphi-\frac{1}{8}\sin^3 2\varphi+\frac{1}{32}\sin^5 2\varphi-
\frac{1}{128}\sin^7 2\varphi =0,$$
or
$$\frac{1}{2} \sin 2\varphi \frac{1-\frac{1}{256}\sin^8 2\varphi}{1+\frac{1}{4}\sin^2 2\varphi} =0.$$
It follows that $\varphi=\pi k/2$, i.e., the square is fixed by Astrelin function with respect to rotation
in its initial position and in the positions which differ from the initial one by an angle multiple to $\pi/2$.
In the same time this quadrangle admits a deformation
\begin{equation}\label{E:square}
\begin{array}{c}
x_1(t)=1-t,\; y_1(t)=0;\quad x_2(t)=0,\; y_2(t)=\sqrt{1+2t-t^2}; \\
x_3(t)=-1+t,\; y_3(t)=0;\quad x_4(t)=0,\; y_4(t)=-\sqrt{1+2t-t^2}
\end{array}
\end{equation}
keeping the lengths of edges, the center of mass, and satisfying Astrelin's condition. So then we can 
affirm that the deformation \eqref{E:square} is a nontrivial flexion of the square (compare with (a)
of Corollary~\ref{C:rigpoly}). At the same time, the fact that the same deformation satisfies the classical
condition $\sum_j x_j y_j=0$ does not ensure its non-triviality. 
\end{example}

\begin{remark}
The hypersurface
$$X_A=\{(x_1,y_1,\dots,x_n,y_n)\in\R^n\,|\,A(x_1,y_1,\dots,x_n,y_n)=0\}\subset\R^{2n}$$
is singular at points where $x_1=y_1=0$ and for each $j\geq 2$ $x_j=0$ or $y_j=0$. But even at its
smooth points $X_A$ can have non-transversal intersections with orbits of $SO(2)$. Indeed, the tangent
vector to the orbit of $SO(2)$ at a point $(x_1,y_1,\dots,x_n,y_n)\in \R^{2n}$
is $(-y_1,x_1,\dots,-y_n,x_n)$. Thus, this point is a non-transversal intersection point of $X_A$ with an
orbit if and only if
$$\begin{cases}
(x_n y_n)^{2n-1}+\cdots+(x_2 y_2)^3+x_1 y_1=0, \\
(2n-1)(x_n y_n)^{2n-2}(x_{n}^{2}-y_{n}^{2})+\cdots+3(x_2 y_2)^{2}(x_{2}^{2}-y_{2}^{2})+
(x_{1}^{2}-y_{1}^{2})=0.
\end{cases}
$$
An example of a non-zero solution to this system of equations can be obtained by letting $x_j=y_j$ for all 
$j=2,\dots,n$ and $x_1=-y_1$, so that the second equation is satisfied, then choosing arbitrary non-zero
values for $x_j=y_j$ for $j=2,\dots,n$, and, finally, choosing $x_1=-y_1$ so that the first equation is also
satisfied. 
\end{remark}

\section{Reduction and fixation of a system of points in space}\label{S:dim3}

The main purpose of this section is to prove Theorem~\ref{T:main}. But we start with a lemma 
describing the non-degenerate orbits of the natural action of $SO(3)$ on $\R^{3n}$, $n\geq 2$.
This lemma is not needed for the proof of Theorem~\ref{T:main}, however, we believe that
it is useful for better understanding of the situation under study.

\begin{lemma}\label{L:nondegorbit}
Assume that not all of the points of a system $\{O\}\cup\{P_j\}_{j=1}^{n}\subset\R^3$, $n\geq 2$,
are collinear. Then the orbit in $\R^{3n}$ of the system $\{P_j\}_{j=1}^{n}$ under the natural
action of the group $SO(3)$ is homeomorphic (in fact, diffeomorphic) to the Lie group $SO(3)$
itself, which in turn is homeomorphic to the projective space $\R\P^3$.
\end{lemma}
\begin{proof}
Suppose that the first two points $P_1$ and $P_2$ of the given system are not collinear with
the origin $O$ of $\R^3$. Recall that we denote $M_0$ the point of $\R^{3n}$ corresponding
to the system $\{P_j\}_{j=1}^{n}$ and denote by $P_{M_0}$ the map
$$P_{M_0}\colon SO(3)\to \R^{3n}, \quad \rho\mapsto \rho\cdot M_0.$$

First let us show that $P_{M_0}$ is an immersion. It follows from the fact that the map $P_{M_0}$ 
commutes with the action of $SO(3)$ (i.e., for any $\rho,\sigma\in SO(3)$ we have 
$P_{M_0}(\rho\sigma)=\rho P_{M_0}(\sigma)$) and with rescaling ($P_{\lambda M_0}=
\lambda P_{M_0}$) that it is enough to check the differential of $P_{M_0}$ at the unity
$1\in SO(3)$ and for $M_0=(1,0,0,a,b,0,\ldots)$, $b\ne 0$. Let us use the parametrization
of $SO(3)$ by Euler angles:
$$\begin{array}{c}
\rho(\alpha,\beta,\gamma)= \\
{\small
\begin{pmatrix}
\cos\alpha \cos\gamma- \cos\beta \sin\alpha \sin\gamma & 
-\cos\gamma \sin\alpha- \cos\alpha \cos\beta \sin\gamma & \sin\beta \sin\gamma \\
\cos\beta \cos\gamma \sin\alpha- \cos\alpha \cos\gamma &
\cos\alpha \cos\beta \cos\gamma- \sin\alpha \sin\gamma & -\cos\gamma \sin\beta \\
\sin\alpha \sin\beta & \cos\alpha \sin\beta & \cos\beta
\end{pmatrix}.}
\end{array}
$$
The unity $1\in SO(3)$ corresponds to $\alpha=\beta=\gamma=0$. A routine calculation reveals that 
the (transpose of) the Jacobi matrix of $P_{M_0}$ at $1$ is
$$J^T=\begin{pmatrix}
0 & 1 & 0 & -b & a & 0 & \dots \\
0 & 0 & 0 & 0  & 0 & b & \dots \\
0 & -1 & 0 & -b & -a & 0 & \dots
\end{pmatrix}.
$$
The minor of $J$ formed by the second, the fourth, and the sixth column is $2b^2\ne 0$, thus
$P_{M_0}$ is indeed an immersion.

Note that the two non-collinear position vectors of the points $P_1$ and $P_2$ determine
by the cross product the third vector which together with the first two forms a basis of $\R^3$. Since
two different rotations $\rho,\sigma\in SO(3)$ cannot act identically on the same basis, it
follows that $\rho\cdot M_0\ne \sigma\cdot M_0$, and hence $P_{M_0}$ is injective.
\end{proof}

We now turn to the proof of Theorem~\ref{T:main}.
\smallskip

\noindent \textit{Proof of Theorem~\ref{T:main}}. Recall that the group  $SU(2)$ is a $2$-fold
covering of the group $SO(3)$. We shall need explicit formulae for this covering. The group $SU(2)$
consists of the matrices
$$\rho=\begin{pmatrix}
\alpha & \beta \\
-\bar{\beta} & \bar{\alpha}
\end{pmatrix}
$$
where $\alpha,\beta\in \C$ и $|\alpha|^2+|\beta|^2=1$. If we represent a point $P_j=
(x_{j}^{\circ},y_{j}^{\circ},z_{j}^{\circ})\in \R^3$ by the matrix
$$H_j=\begin{pmatrix}
z_{j}^{\circ} & x_{j}^{\circ}+i y_{j}^{\circ} \\
x_{j}^{\circ}-i y_{j}^{\circ} & -z_{j}^{\circ}
\end{pmatrix},
$$
then the action of the $3$-dimensional rotation corresponding to the matrix $\rho$ on the point $P_j$
can be computed as
$$H_j\mapsto \rho H_{j} \rho^{-1}$$
(see \cite[Chapter 7 \S 1.3]{Kostrikin}). A direct calculation shows that after the rotation $\rho$
the projection of $P_j$ to the plane $Oxy$ (represented as a complex number) is
$$w_j=c_j \alpha^2-\bar{c}_{j} \beta^2-2z_{j}^{\circ} \alpha \beta,$$
where $c_j=x_{j}^{\circ}+i y_{j}^{\circ}$. In particular, the projection of $P_j$
to the plane $Oxy$ as a function of rotation $\rho$ is a restriction of an \emph{analytic function}
(quadratic form) of variables $\alpha$ and $\beta$ from the space $\C^2$ to the unit $3$-dimensional
sphere
$$S^3=\{(\alpha,\beta)\in \C^2 \,|\, |\alpha|^2+|\beta|^2=1\}\simeq SU(2).$$

\begin{lemma}\label{L:not0}
The function $F$ does not identically vanish on any of the orbits of $SO(3)$ on $\R^{3n}$ with
the exception of the trivial case when the orbit reduces to the single point $0\in\R^{3n}$.
\end{lemma}

\begin{proof}
Suppose that at least one of the points of the system $\{P_j\}_{j=1}^{n}$ is different from the
origin. Without loss of generality we can assume that it is $P_1$.
\smallskip

\noindent \emph{Case I}: the points $O,P_1,\dots,P_n$ lie on the same line. Then, the coordinates
of all the points $P_2,\dots,P_n$ are proportional to the coordinates of the point $P_1$, and the same
holds for the projections of the points to the plane $Oxy$:
$$c_j=\lambda_j c_1,\: \lambda_j\in\R,\: j=2,\dots,n.$$
This relation is preserved under each rotation of the system around the origin. Substituting
$w_j=\lambda_j w_1$ to the function $F$, we get
$$F(w_1,\lambda_2 w_1,\dots,\lambda_n w_1)=w_{1}^{2d} F(1,\lambda_2,\dots,\lambda_n).$$
Since the hypersurface $X_F=\{F=0\}$ has no real points, the last expression can vanish only if
$w_1=0$, i.e., if all points are disposed on the axis $Oz$.
\smallskip

\noindent \emph{Case II}: the points $O,P_1,\dots,P_n$ don't lie on the same line, but they are
in the same plane. Without loss of generality we assume that the position vectors of $P_1$ and $P_2$ 
are linearly independent and the rest are their linear combinations. The same dependence holds also
for the projections and is preserved under any rotation of the system of points around the origin.
Let
$$c_j=\lambda_j c_1+\mu_j c_2,\:\lambda_j,\mu_j\in\R,\: j=3,\dots,n.$$
The quadratic forms
\begin{align*}
w_1 &=c_1 \alpha^2-\bar{c}_{1} \beta^2-2z_{1}^{\circ} \alpha \beta, \\
w_2 &=c_2 \alpha^2-\bar{c}_{2} \beta^2-2z_{2}^{\circ} \alpha \beta
\end{align*}
are also linear independent, and thus define a $4$-fold ramified covering of the complex
$2$-dimensional $w$-plane $\C^2$ by the $(\alpha,\beta)$-plane $\C^2$. The formulae
\begin{align*}
w_1 &=c_1 \alpha^2-\bar{c}_{1} \beta^2-2z_{1}^{\circ} \alpha \beta, \\
w_2 &=c_2 \alpha^2-\bar{c}_{2} \beta^2-2z_{2}^{\circ} \alpha \beta, \\
w_j &=\lambda_j w_1+\mu_j w_2,\: j=3,\dots,n,
\end{align*}
define a map of the $(\alpha,\beta)$-plane $\C^2$ onto a $2$-dimensional complex linear subspace
$L$ of the space $\C^n$. If the function $F$ vanished identically on the orbit of the system
$\{P_j\}_{j=1}^{n}$, then it would vanish identically also on the subspace $L$. Indeed,
the set of zeros of an analytic function on $\C^2$ has real dimension $2$ or $4$ (or is empty,
if the function is a non-zero constant), thus since our function vanishes on the $3$-dimensional
unit sphere $S^3\subset \C^2$, it must vanish on the all of $\C^2$. Passing to the
projectivization, vanishing of $F$ on $L$ would mean that the hypersurface $X_F\subset
\mathbf{CP}^{n-1}$ contained a line, which would contradict to our assumptions.
\smallskip

\noindent \emph{Case III}: the points $O,P_1,\dots,P_n$ do not lie in one plane. As before,
we assume from the beginning that the position vectors of the points $P_1,P_2,P_3$ are linearly
independent and
$$c_j=\lambda_j c_1+\mu_j c_2+\nu_j c_3$$
for some $\lambda_j,\mu_j,\nu_j\in\R$,  $j=4,\dots,n$. Then, the quadratic forms
\begin{align*}
w_1 &=c_1 \alpha^2-\bar{c}_{1} \beta^2-2z_{1}^{\circ} \alpha \beta, \\
w_2 &=c_2 \alpha^2-\bar{c}_{2} \beta^2-2z_{2}^{\circ} \alpha \beta, \\
w_3 &=c_3 \alpha^2-\bar{c}_{3} \beta^2-2z_{3}^{\circ} \alpha \beta
\end{align*}
are also linear independent and define a map of the $(\alpha,\beta)$-plane $\C^2$ onto
a non-degenerate (i.e., not splitting into $2$ planes) quadratic cone in the $3$-dimensional
complex $w$-space $\C^3$. Note that this is nothing else but the affine Veronese map from
$\C^2$ to $\C^3$. Formulae
\begin{align*}
w_1 &=c_1 \alpha^2-\bar{c}_{1} \beta^2-2z_{1}^{\circ} \alpha \beta, \\
w_2 &=c_2 \alpha^2-\bar{c}_{2} \beta^2-2z_{2}^{\circ} \alpha \beta, \\
w_3 &=c_3 \alpha^2-\bar{c}_{3} \beta^2-2z_{3}^{\circ} \alpha \beta, \\
w_j &=\lambda_j w_1+\mu_j w_2+\nu_j w_3,\: j=4,\dots,n,
\end{align*}
define a map from the $(\alpha,\beta)$-plane $\C^2$ onto a $2$-dimensional non-degenerate
quadratic cone $Q$ in the space $\C^n$. If the function $b$ vanished identically on the orbit
of the system of points $\{P_j\}_{j=1}^{n}$, then it would vanish identically also on the cone $Q$.
But then the projectivization of $Q$ (i.e., a non-degenerate conic) would be contained in the
hypersurface $X_b\subset\mathbf{CP}^{n-1}$, which is again impossible due to our assumptions.
\end{proof}

Let us continue the proof of Theorem~\ref{T:main}. Consider the cases analogous to the proof
of Lemma~\ref{L:not0}. If all the points of the system are concentrated at the origin, then the
theorem is obvious. In \emph{Case I} we saw that both the functions $\re F$, $\im F$ (and thus
also the function $F$) vanish if and only if all the points lie on the coordinate axis $Oz$.
Already this provides fixation of the system. But for such position of points the function $\im H$
vanishes as well, so we have reduction too.

In \emph{Cases II and III} let us denote by $\tilde{F}$ and $\tilde{H}$ respectively the functions 
obtained from $F$ and $H$ by the substitution
\begin{equation}\label{E:su2tocn}
w_j=c_j \alpha^2-\bar{c}_{j} \beta^2-2z_{j}^{\circ} \alpha \beta,\: j=1,\dots,n.
\end{equation}
They are polynomials of $2$ complex variables $\alpha$ and $\beta$. The polynomial $\tilde{F}$ 
is homogeneous of degree $4d$ and, as we checked in Lemma~\ref{L:not0}, does not identically vanish. 
It follows that it decomposes into linear factors. In other words, the set of zeros of $\tilde{F}$ on $\C^2$ 
is a union of $4d$ (counted with multiplicities) complex $1$-dimensional subspaces (lines). The map 
from $\C^2$ to $\C^n$ defined by formulae~\eqref{E:su2tocn} transforms such lines into 
$1$-dimensional subspaces (lines) of the space $\C^n$. On the other hand, the function $H$ does 
not identically vanish on any of the complex lines passing through $0\in\C^n$ of the space $\C^n$. 
Hence, also the function $\tilde{H}$ does not identically vanish on any of the lines of the space $\C^2$. 
Now Theorem~\ref{T:main} can be deduced from the following simple fact.

\begin{lemma}\label{L:zerosofim}
Let $f(z)$ be a non-constant complex polynomial considered as a function $f\colon\C\to \C$.
Assume also that $f(0)=0$. Then the set of points of the unit circle
$S^1=\{z\in \C\,|\,|z|=1\}$
at which $f$ takes real values is not empty and finite.
\end{lemma}
\begin{proof}
The proof follows from standard theorems of analysis, but we provide it for the sake of completeness.
If the polynomial $f$ took no real values on the unit circle, its imaginary part $\im f$ would be
either strictly positive or strictly negative on the circle. But $\im f=0$ at $z=0$. It follows
that the harmonic function $\im f$ would attain its minimum (or maximum) inside the
open unit disc, but that would violate the maximum principle for $\im f$.

Now assume that $\im f$ vanishes on an infinite set of points of the unit circle.
Both the unit circle $S^1$ and the zero set of $\im f$ are real algebraic sets (because $f$ is
a polynomial) and, moreover, the circle is irreducible and remains irreducible after complexification. 
It follows that $S^1\subseteq (\im f)^{-1}(0)$, i.e., the harmonic polynomial $\im f$ vanishes 
everywhere on the circle $S^1$. Again by the maximum principle for harmonic functions we can
claim that $\im f=0$ on the whole unit disc and, since $\im f$ is a polynomial, on the whole
complex plane $\C$. In other words, the complex polynomial $f$ takes only real values. 
But since $f$ is non-constant, this is a contradiction with, say, openness of an analytic map. 
\end{proof}

To finish the proof of Theorem~\ref{T:main}, apply Lemma~\ref{L:zerosofim} to the restrictions
of the polynomial $\tilde{H}$ to each of the lines that constitute the set of zeros of the polynomial
$\tilde{F}$. \qed

\begin{remark}
It is clear from the proof of Theorem~\ref{T:main} that the choice $f_3=\im H$ can be changed
to $f_3=\re H$. It is clear also that if one is interested in reduction and fixation of
\emph{non-degenerate} system of points (in the sense that $O,P_1,\ldots,P_n$ are
not contained in a plane), then the conditions about no lines and no real points on $X_F$
can be dropped, and a condition only needed is that $X_F$ has no conics. In particular,
there is no need for $X_F$ to be of even degree. Furthermore, also the condition about
the conics can be somewhat weakened. A closer look at the projective conic defined by
parametric equations~\eqref{E:su2tocn} reveals that it has no real points. Thus, the
sufficient condition on $F$ is that $X_F$ has no conics \emph{with real points}.
\end{remark}

\begin{example}
A system $\{P_1\}$ with one point can be reduced (by $SO(3)$) and fixed by only two equations
$x_1=y_1=0$. Let $n=2$. A system with two points is always degenerate. The conditions
about lines and conics in Theorem~\ref{T:main} become empty, so we can set
$F=w_{1}^{2}+w_{2}^{2}$, $G=w_{1}^{2}+w_2$. But a much simpler method not
relying on Theorem~\ref{T:main} is to impose the equations $z_1=z_2=0$, thus
embedding the system to the coordinate plane $Oxy$, and add one more
equation $(x_1 y_1)^3+x_2 y_2=0$ with Astrelin function, reducing to the plane case.
\end{example}

\begin{example}\label{Ex:n34}
The first non-trivial case is that of $n=3$ points. The hypersurface $X_F$ in this case is
a curve, so it suffices to ensure that it is not a conic, is irreducible and has no real points. The choice
$$F=w_{1}^{4}+w_{2}^{4}+w_{3}^{4}$$
defining the Fermat quartic and
$$H=w_{1}^{3}+w_{2}^{2}+w_1$$
works well. We can deal with the case of $n=4$ points if the group $G$ is extended to the full affine
group of isometries of $\R^3$. First we impose the conditions \eqref{E:cofm}
fixing the center of mass. Then, only three of the points remain independent,
and the situation essentially reduces to the case $n=3$. We again can choose
$$F=w_{1}^{4}+w_{2}^{4}+w_{3}^{4}+w_{4}^{4}$$
to be the Fermat quartic. A straightforward verification shows that the section of
$X_F\subset \C\P^4$ by the plane
$$w_1+w_2+w_3+w_4=0$$
is a non-singular (and thus irreducible) plane quartic with no real points. Thus this $F$
and, say,
$$H=w_{1}^{4}+w_{2}^{3}+w_{3}^{2}+w_4$$
provide reduction and fixation of a system with $4$ points with respect to the full group of affine
isometries of $\R^3$. 
\end{example}

\begin{example}
It is still possible to write explicit function $F$ for reduction and fixation of a system of $4$
points with respect to the group $SO(3)$ only. In \cite[Theorem~3.1]{vanLuijk}, van Luijk describes
a family of quartics in $\C\P^3$ defined over $\Q$ and of Picard number $1$. The last condition
means, in particular, that each of the surfaces has no lines and no conics. The equation of
such a quartic $X_h$ is
\begin{equation}\label{E:vanLuijk}
w_4 f_1+2w_3f_2=3g_1 g_2+6h,
\end{equation}
where $w_1,\ldots,w_4$ are the homogeneous coordinates on $\C\P^3$, $f_1$, $f_2$, $g_1$, $g_2$
are some explicitly given homogeneous polynomials whose precise form is not important here, $f_i$
of degree $3$, $g_j$ of degree $2$, and $h$ is \emph{any} homogeneous polynomial of degree $4$.
It follows that we can choose $h$ to be the Fermat sum of $4$-th degrees of $w_1,\ldots,w_4$
taken with a very big integral coefficient so that the resulting quartic $X_h$ has no real points. Thus
the function $F$ defining this quartic and $H$ from Example~\ref{Ex:n34} provide reduction and
fixation of any system of $4$ points in space. In fact, a version of the trick from Example~\ref{Ex:n34}
applied to $X_h$ allows to deal also with systems of $5$ points and with the group $G$ the full group of 
affine isometries. For this, we take
$$F(w_1,\dots,w_5)=w_4 f_1+2w_3f_2-3g_1 g_2-6h$$
with
$h=-N(w_{1}^{4}+\dots+w_{4}^4+w_{5}^{4})$.
In view of the equations fixing the center of mass, the problem amounts to study of the quartic
in $\C\P^3$ with the equation
$$w_4 f_1+2w_3f_2-3g_1 g_2+6N(w_{1}^{4}+\dots+w_{4}^4+(w_1+\dots+w_4)^4)=0.$$
It has the same form as van Luijk's equation \eqref{E:vanLuijk}, thus defines a quartic with no
lines and conics, and for sufficiently big $N$ it has no real points.
\end{example}

\begin{remark}
In fact, Lemma~\ref{L:zerosofim} gives a method for constructing functions providing
rotational reduction and fixation of systems of points in the plane different from Astrelin's function $A$.
Indeed, let us represent each point $P_j=(x_j,y_j)$ of a system $\{P_j\}_{j=1}^{n}\subset \R^2$
by the complex number $w_j=x_j+iy_j$. Let $H$ be any polynomial on $\C^n$ such that
$H(0,\ldots,0)=0$ but $H$ does not identically vanish on any of the $1$-dimensional vector
subspaces of $\C^n$. An example of such $H$ is given in Theorem~\ref{T:main}. Then each
system $\{P_j\}_{j=1}^{n}$ in the plane admits a reduction by $SO(2)$ with respect to the
function $\im H$ and is fixed by this function with respect to $SO(2)$.
\end{remark}

\begin{example}
Let $\{P_j\}_{j=1}^{n}$ be the set of vertices of a regular polygon in $\R^2$ ($\C$), and
the initial positions of the points $P_j$ are the $n$th complex roots of unity:
$$P_j=e^{2pi(j-1)/n}, \quad j=1,\dots,n.$$
Let us choose the function $H$ this time to be
$$H=w_{1}^{n}+w_{2}^{2n}+\dots+w_{n}^{n^2}.$$
The action of $SO(2)$ can now be represented as multiplication of each $w_j$ by a complex number
$e^{\varphi i}$. The reader can easily calculate that the restriction of $\im H$ to the $SO(2)$-orbit of 
the system $\{P_j\}_{j=1}^{n}$ is
$$\im\left(\sum_{j=1}^{n} \big(e^{2\pi i(j-1)/n} e^{i\varphi}\big)^{jn}\right)=
\im\left(\sum e^{ijn\varphi}\right)=\sum_{j=1}^{n} \sin j\psi,$$
where $\psi=n\varphi$. The last sum is $0$ if $\psi=2\pi k$, and, if $\psi\ne 2\pi k$, it can be expressed 
as
$$\frac{\sin\frac{n\psi}{2}\sin\frac{(n+1)\psi}{2}}{\sin\frac{\psi}{2}}.$$
It follows that the regular polygon is fixed by $\im H$ at its initial position and at positions
that differ from the initial one by a rotation by an angle $\varphi=\frac{2\pi k}{n^2}$, $k=1,\ldots,n^2-1$,
or $\varphi=\frac{2\pi k}{n(n+1)}$, $k=1,\ldots,n^2+n-1$.
\end{example}

The fact that the Astrelin type function $H$ is not homogeneous has a consequence that the rotation
providing reduction is not invariant under scaling, i.e., multiplication of all the coordinates of all the points
of a system by the same number. This should seem very unnatural from physicist's point of view.
As our last remark we shall show that it is possible to replace $H$ by a homogeneous function,
but at the price of making the construction even less explicit.

\begin{theorem}
Let $F$ be a complex homogeneous polynomial satisfying the conditions of Theorem~\ref{T:main},
and $g(x_1,y_1,\ldots,x_n,y_n)$ be a \emph{real} homogeneous polynomial of \emph{odd} degree
such that its complexification $g_{\C}\colon \C^{2n}\to \C$ defines a hypersurface
$X_{g_{\C}}\subset \C\P^{2n-1}$ which \emph{has no lines}. Then each system of $n$ points in
$\R^3$ admits a reduction by $SO(3)$ with respect to the $3$ functions  
$$f_1=\re F, \: f_2=\im F, \: f_3=g,$$
considered as functions on $\R^{3n}$, and at the same time is fixed by $f_1$, $f_2$, $f_3$ with
respect to $SO(3)$.
\end{theorem}
\begin{proof}
The proof starts with the same argument as the proof of Theorem~\ref{T:main} till we highlight
several complex lines in $\C^n$, and the question reduces to proving that the restriction
of $g$ to each of these lines (considered now as real planes in $\R^{2n}$) has a finite number
of zeros on the unit circle. But indeed, the condition that we imposed on $g$ ensures that $g$
is not identically zero on any such plane $L$. On the other hand, the restriction $g|_L$ is
a homogeneous function of odd degree. Thus $g|_L$ vanishes along several real lines, and these
lines intersect the unit circle in a finite set of points.
\end{proof}

\end{document}